\documentclass[smallextended,envcountsect]{svjour3}

\usepackage{epstopdf}
\usepackage{graphicx,amsmath,amssymb,setspace}
\usepackage{amsthm}
\usepackage{srcltx}
\usepackage{graphicx}
\allowdisplaybreaks
\newtheorem{thm}{Theorem}[section]
\newtheorem{cor}{Corollary}[section]
\newtheorem{prop}{Proposition}[section]

\newcommand{\pv}{\pi^*}
\newcommand{\alp}{\alpha}
\newcommand{\C}{\mathcal{C}}

\newcommand{\E}{\mathbb{E}}
\newcommand{\R}{\mathbb{R}}
\renewcommand{\P}{\mathbb{P}}
\newcommand{\F}{\mathcal{F}}
\newcommand{\hpsi}{\hat \psi}%

\newcommand{\yo}{y_0}
\newcommand{\yL}{y_L}
\newcommand \yoL {y_{0L}}

\newcommand{\la} {\lambda}
\newcommand{\sig} {\sigma}
\newcommand \del {\delta}

\newcommand \kap {\kappa}
\newcommand \bet {\beta}

\newcommand \dd {{\rm d}}

\newcommand {\psio} {\psi^{(0)}}
\newcommand {\psil} {\psi^{(1)}}
\newcommand \De {{\bf D}}
\newcommand \K {{\bf K}}

\journalname{JOTA}

\begin{document}

\title{Minimizing the Probability of Lifetime Exponential Parisian Ruin}


\author{Xiaoqing Liang   \and  Virginia R. Young }
\institute{Xiaoqing Liang \at
             Hebei University of Technology \\
              Tianjin, China\\
              liangxiaoqing115@hotmail.com
           \and
           Virginia R. Young,  Corresponding author  \at
              University of Michigan \\
              Ann Arbor, Michigan\\
              vryoung@umich.edu
}

\date{Received: date / Accepted: date\\ Communicated by {Dylan Possama\"{i}}}

\maketitle

\begin{abstract}
We find the optimal investment strategy in a Black-Scholes market to minimize the probability of so-called {\it lifetime exponential Parisian ruin}, that is, the probability that wealth exhibits an excursion below zero of an exponentially distributed time before the individual dies.  We find that leveraging the risky asset is worse for negative wealth when minimizing the probability of lifetime exponential Parisian ruin than when minimizing the probability of lifetime ruin.  Moreover, when wealth is negative, the optimal amount invested in the risky asset increases as the hazard rate of the exponential ``excursion clock''  increases.  In view of the heavy leveraging when wealth is negative, we also compute the minimum probability of lifetime exponential Parisian ruin under a constraint on investment.  Finally, we derive an asymptotic expansion of the minimum probability of lifetime exponential Parisian ruin for small values of the hazard rate of the excursion clock.  It is interesting to find that, for small values of this hazard rate, the minimum probability of lifetime exponential Parisian ruin is proportional to the minimum occupation time studied in Bayraktar and Young, and the proportion equals the hazard rate.  To the best of our knowledge, our work is the first to {\it control} the probability of Parisian ruin.
\end{abstract}

\keywords{Exponential Parisian ruin \and Optimal investment \and Stochastic control}

\subclass{93E20 \and  91B30 \and  49K10 \and  49L20}


\section{Introduction}

The probability of lifetime ruin measures the risk that individuals run out of money before they die, which was first proposed by Milevsky and Robinson \cite{1.} and extended by Young \cite{2.}, who considered how individuals optimally invest in a risky financial market to minimize this probability.  This problem was extended to subsequent variants based on Young's work, including adding borrowing constraints (Bayraktar and Young \cite{3.}), assuming model ambiguity in the drift of the risky asset (Bayraktar and Zhang \cite{4.}), allowing stochastic volatility (Bayraktar et al.\ \cite{5.}), and transaction costs (Bayraktar and Zhang \cite{6.} and Liang and Young \cite{7.}).

From a practical point of view, investors might be able to sustain negative wealth until some (random) time when they are called to be accountable for their bankruptcy or until their wealth recovers to positive territory.  For example, during the financial crisis of the late 2000s, many homeowners were ``under water,'' meaning that the value of their home was less than the outstanding mortgage, and some of these homeowners  were able to maintain their negative wealth position until they got ``above water'' again.  On the other hand, many of these homeowners had to declare bankruptcy for a variety of reasons, such as loss of employment or increased interest rates that they could not afford.  One way to address the problem of how to invest when the wealth is negative is to minimize the expected time that the individual's wealth stays below zero, which is called the occupation time, as studied by Bayraktar and Young \cite{8.}.  Another way is to minimize the probability of lifetime exponential Parisian ruin, and that is the problem addressed in this paper.

Inspired by Parisian options (see Chesney et al.\ \cite{9.}), Parisian ruin occurs if an excursion below zero is longer than a deterministic time. Dassios and Wu \cite{10.} computed the Parisian ruin probability for a classical risk model with exponential claims and for the Brownian motion with drift.  Recently, Czarna and Palmowski \cite{11.} and Loeffen et al.\ \cite{12.} considered the probability of Parisian ruin for a spectrally negative L\'evy process using the tool of scale functions. Landriault et al.\  \cite{13.,14.} studied the Laplace transform of Parisian ruin with stochastic implementation delays. Gu\'erin and Renaud \cite{15.} computed the distribution of cumulative Parisian ruin.

In this paper, we consider the problem of optimally investing to minimize the probability of lifetime exponential Parisian ruin, that is, the probability wealth stays below zero longer than an exponentially distributed time and before the individual dies.\footnote{To the best of our knowledge, our work is the first to {\it control} the probability of Parisian ruin; all the other related research we found focused on calculating the probability of Parisian ruin.}  The individual consumes and invests in a Black-Scholes financial market consisting of one riskless and one risky asset, whose price process follows a geometric Brownian motion.  We find that individuals leverage more when wealth is negative when minimizing the probability of lifetime exponential Parisian ruin than when minimizing the probability of lifetime ruin.  Moreover, when wealth is negative, the optimal amount invested in the risky asset increases as the hazard rate of the exponential ``excursion clock''  increases.  In view of the heavy leveraging when wealth is negative, we also compute the minimum probability of lifetime exponential Parisian ruin under a constraint on investment.  Finally, we derive an asymptotic expansion of the minimum probability of lifetime exponential Parisian ruin for small values of the hazard rate of the excursion clock.  It is interesting to find that, for small values of this hazard rate, the minimum probability of lifetime exponential Parisian ruin is proportional to the minimum occupation time studied in Bayraktar and Young \cite{8.}, and the proportion equals the hazard rate of the excursion clock.

The remainder of the paper is organized as follows.  In Section 2, we describe the financial market in which the individual consumes and invests, we formalize the problem of minimizing the probability of lifetime exponential Parisian ruin, we state a verification lemma that will enable us to find that minimum probability, along with the optimal strategy for investing in the financial market, and we solve the problem of minimizing the probability of lifetime exponential Parisian ruin.  In Section 3, we analyze the solution found in Section 2 and provide a numerical experiment to illustrate the results of that section.  In Section 4, we address two additional considerations: (1) limit the investment strategy so that the amount invested in the risky asset is no greater than when minimizing the probability of lifetime ruin; and (2) provide an asymptotic expansion of the minimum probability of lifetime exponential Parisian ruin for small values of the hazard rate of the excursion clock.  Section 5 concludes the paper.

\section{Minimizing the Probability of Lifetime Exponential Parisian Ruin}\label{sec:Paris-ruin}

In Section \ref{sec:fin-model}, we describe the financial market in which the individual invests her wealth, and we formulate the problem of minimizing the probability of lifetime exponential Parisian ruin.  In that section, we also provide a verification theorem for the minimum probability of lifetime exponential Parisian ruin.  Then, in Section \ref{sec:FBP}, we construct the minimum probability as the convex Legendre dual of the solution of a related free-boundary problem.

{\bf \subsection{Financial Model and Verification Theorem}}\label{sec:fin-model}

We consider an individual with future lifetime given by the random variable $\tau_d$ living on a probability space $(\Omega, \F, \P)$.  Suppose $\tau_d$ is an exponential random variable with hazard rate $\la$, also referred to as the force of mortality; in particular, $\E(\tau_d) = 1/\la$.\footnote{The assumption of constant force of mortality rate $\la$ enables us to obtain explicit solutions.  See the work in Moore and Young \cite{16.} concerning using this assumption to obtain simple and nearly-optimal investment strategies when minimizing the probability of lifetime ruin.}

We assume that the individual consumes wealth at a constant {\it net} rate of $c$; this rate may be given in real or nominal units.  We say that the rate $c$ is a net rate because it is the rate of consumption offset by any income. One can interpret $c$ as the minimum net consumption level below which the individual cannot (or will not) reduce her consumption; therefore, the minimum probability of lifetime Parisian ruin that we compute gives a lower bound for the probability of lifetime Parisian ruin under any rate of consumption bounded below by $c$.

The individual can invest in a riskless asset, which earns interest at the rate $r > 0$.  Also, she can invest in a risky asset whose price process follows
\begin{equation}
\dd S_t = \mu S_t \dd t + \sig S_t \dd B_t, \quad S_0 = S > 0,
\end{equation}
in which $\mu > r$, $\sig > 0$,\footnote{If $c$ were given as a real rate of consumption (that is, after inflation), then we would also express $r$ and $\mu$ as real rates.} and $B$ is a standard Brownian motion with respect to a filtration $\mathbb{F} = \{{\F}_t \}_{t \geq 0}$ of the probability space $(\Omega, \F, \P)$.  We assume that $B$ is independent of $\tau_d$, the random time of death of the individual.  We enlarge the filtration $\mathbb{F}$ to include the information generated by the death process $D$.  Specifically, define the death process $D = \{D_t\}_{t \ge 0}$ by $D_t = {\bf 1}_{\{\tau_d \le t\}}$; thus, $D$ jumps from 0 to 1 when the individual dies. Then, let $\mathbb{G} = \{ \mathcal{G}_t \}_{t \ge 0}$ be the progressive enlargement of the filtration $\mathbb{F}$ by $D$, in which $\mathcal{G}_t = \F_t \vee \sig(D_u: 0 \le u \le t)$ for all $t \ge 0$.  Assume $\mathbb{F}$ and $\mathbb{G}$ are augmented to satisfy the usual conditions of completeness and right continuity.  Note that $B$ is a $\mathbb{G}$-martingale.

Let $\pi_t$ denote the amount invested in the risky asset at time $t$, and let $\pi$ denote the investment strategy $\{\pi_t\}_{t \geq 0}$.  We say that a strategy $\pi$ is {\it admissible} (1) if the process $\pi$ is adapted to the filtration $\mathbb{F}$, (2) if $\pi$ satisfies $\int_0^t \pi_s^2 \, \dd s < \infty$, almost surely, for all $t \ge 0$, and (3) if, given $W_t = c/r$, then $\pi_s = 0$ for all $s \ge t$.  The third condition ensures that if wealth reaches $c/r$, then it never falls below $c/r$.  The wealth dynamics of the individual under an admissible strategy $\pi$ are given by
\begin{equation}\label{eq:wealth}
\dd W^\pi_t = (r W_t + (\mu - r) \pi_t - c) \dd t + \sig \pi_t \dd B_t, \quad
W^\pi_0 = w.
\end{equation}

By {\it lifetime exponential Parisian ruin}, we mean an excursion of wealth below zero in excess of a random length of time before the individual dies, in which the random length of time is exponentially distributed.  One could also consider an excursion of wealth below some arbitrary level, not necessarily $0$, but for ease of presentation, we choose the level to be $0$.  By following Gu\'erin and Renaud \cite{15.}, we define the time of exponential Parisian ruin by
\begin{equation}\label{eq:expPar_ruin}
\kap^\pi = \inf \{t > 0: t - g_t > {\bf e}_\rho^{g_t} \},
\end{equation}
in which $g_t = \sup \{ s \in [0, t]: W^\pi_s \ge 0 \}$  and each random variable ${\bf e}_\rho^{g_t}$ is exponentially distributed with hazard rate $\rho > 0$.  Note that we reset the ``excursion clock'' to $0$ whenever surplus reaches $0$ from below; indeed, if $W^\pi_t \ge 0$, then $t - g_t = 0$, and $\kap^\pi = \inf \emptyset = \infty$.  Note that $\kap^\pi$ depends on the investment strategy $\pi$ via the wealth process $W^\pi$.

For each strategy $\pi$, we further augment the filtration $\mathbb{G}$ to include information generated by $\kap^\pi$.  Specifically, define the  exponential Parisian ruin process $K^\pi = \{K^\pi_t\}_{t \ge 0}$ by $K^\pi_t = {\bf 1}_{\{\kap^\pi \le t\}}$ for all $t \ge 0$; thus, $K^\pi$ jumps from 0 to 1 when exponential Parisian ruin occurs.  Then, we define the filtration $\mathbb{H}^\pi = \{ \mathcal{H}^\pi_t \}_{t \ge 0}$ by $\mathcal{H}^\pi_t = \mathcal{G}_t \vee \sig(K^\pi_u: 0 \le u \le t)$ for all $t \ge 0$.  Note that $B$ is a $\mathbb{H}^\pi$-martingale.

We wish to minimize the probability of lifetime exponential Parisian ruin, defined by
$$
\P^w \big( \kap^\pi < \tau_d \big),
$$
in which $\P^w$ denotes probability conditional on $W_0 = w$.  However, there is a problem with the goal of minimizing this probability.  Indeed, we expect the minimum of $\P^w \left( \kap^\pi < \tau_d \right)$ to be a bounded, convex, non-increasing function of initial wealth $w$, but there is no bounded, convex, non-increasing function defined on the reals, other than a constant function.  Therefore, we modify the problem as follows: First, define the minimum wealth process $Z^\pi = \{ Z^\pi_t \}_{t \ge 0}$ by
\[
Z^\pi_t = \inf \limits_{0 \le s \le t} W^\pi_s.
\]
Then, for a positive constant $L$, define the value function $\psi$ by
\begin{equation}\label{eq:psi}
\psi(w) = \inf_\pi \E^w \Bigg( {\bf 1}_{\{\kap^\pi < \tau_d \}} {\bf 1}_{\{ Z^\pi_{\kap^\pi \wedge \tau_d} > -L \}} + \dfrac{\rho}{\la + \rho} \, {\bf 1}_{\{ Z^\pi_{\kap^\pi \wedge \tau_d} \le -L \}} \Bigg).
\end{equation}
Here, we take the infimum over admissible investment strategies.

For large values of $L > 0$, the control problem associated with $\psi$ approximates the problem of minimizing the probability of lifetime exponential Parisian ruin.  Indeed, if lifetime wealth stays above $-L$, which is likely for $L$ large, then the payoff is the probability of lifetime exponential Parisian ruin.  If wealth falls below $-L$, then we suppose that the individual will have negative wealth for the rest of her life (that is, until $\tau_d$) or until the exponential excursion clock runs down (that is, until $\tau_\rho$), and we end the game with the value of $\frac{\rho}{\la + \rho}$, which equals the probability that a generic $\tau_\rho$ occurs before $\tau_d$.

Note that for $w \le -L$, $\psi(w) = \frac{\rho}{\la + \rho}$, and for $w \ge c/r$, $\psi(w) = 0$.  The latter holds because if $w \ge c/r$, then the individual can place all her wealth in the riskless asset and wealth will never go below $w \ge c/r$, much less reach 0 and trigger possible exponential Parisian ruin.  Therefore, it remains for us to determine $\psi$ for $w \in [-L, c/r]$, and a verification theorem will help us with that goal.

We present the following verification theorem, whose proof closely follows the one of Theorem 2.1 in Bayraktar and Young \cite{8.}.

\begin{thm}\label{thm:verif}
Suppose $\Psi: [-L, c/r] \to [0, 1]$ is a continuous function that satisfies the following conditions.
\begin{enumerate}
\item{} $\Psi$ is non-increasing and convex and lies in $\C^2\big(\; ]-L, c/r[ \, \big),$ except at $0,$ where it is $\C^1$ and has left- and right-second derivatives.
\item{} $\Psi(-L) = \dfrac{\rho}{\la + \rho} \, $.
\item{} $\Psi(c/r) = 0$.
\item{} $\Psi$ solves the following Hamilton-Jacobi-Bellman $($HJB$)$ equation on $]-L, c/r[\,:$
\begin{equation}\label{eq:HJB}
\la \Psi + \rho( \Psi - 1) {\bf 1}_{\{w < 0\}} = (rw - c) \Psi_w + \inf \limits_\pi \left[(\mu-r)\pi \Psi_w + \frac{1}{2} \, \sig^2 \pi^2 \Psi_{ww} \right].
\end{equation}
\end{enumerate}
Then, the value function $\psi$ on $[-L, c/r]$ defined by \eqref{eq:psi} is given by
\begin{equation}\label{eq:verif}
\psi(w) = \Psi(w),
\end{equation}
and the optimal investment strategy $\pv$ on $]-L, c/r[$ is given in feedback form by
\begin{equation}\label{eq:pi}
\pv_t = - \, \dfrac{\mu - r}{\sig^2} \, \dfrac{\Psi_w \big(W^*_t \big)}{\Psi_{ww} \big(W^*_t \big)} \, ,
\end{equation}
in which $W^*_t$ is the optimally controlled wealth at time $t$.
\end{thm}

\begin{proof}
Assume that $\Psi$ satisfies the conditions specified in the statement of this theorem. Let $N^d$ denote a Poisson process with rate $\la$ that is independent of the standard Brownian motion $B$ driving the wealth process. The occurrence of a jump in $N^d$ represents the death of the individual.

Let $\pi: \R \to \R$ be a function, and let $W^\pi$ and $Z^\pi$ denote the wealth and minimum wealth, respectively, when the individual follows the investment policy $\pi_t = \pi(W_t)$.  Assume that this investment policy is admissible.  Let $N^\rho$ denote a Poisson process with rate $\rho$ that is "observed" whenever $W^\pi$ is negative.  If $N^\rho$ jumps (while $W^\pi$ is negative), then exponential Parisian ruin has occurred.

Define two additional states $\De$ and $\K$.  The wealth process is killed (and sent to $\De$) as soon as the Poisson process $N^d$ jumps (that is, when the individual dies), and we assign $W^\pi_{\tau_d} = \De$; recall that $\tau_d$ is the random time of death.  Similarly, the wealth process is killed (and sent to $\K$) as soon as the Poisson process $N^\rho$ jumps (that is, when exponential Parisian ruin occurs), and we assign $W^\pi_{\kap^\pi} = \K$; recall that $\kap^\pi$ is the random time of exponential Parisian ruin.  Given a function $g$ on $[-L, c/r]$, we extend it to the states $\De$ and $\K$ by defining $g(\De) = 0$ and $g(\K) = 1$.

Define the stopping time $\tau = \tau_d \wedge \kap^\pi \wedge \tau_{c/r} \wedge \tau_L$, in which $\tau_{c/r} = \inf\{t \ge 0: W^\pi_t \ge c/r \}$ and $\tau_L = \inf\{ t \ge 0: W^\pi_t \le -L \}$.  Also, define $\tau_n = \inf \{t \ge 0:  \int_0^t \pi^2_s \, \dd s \ge n \}$.  By applying a generalized It\^o's formula for convex functions to $\Psi \big(W^\pi_{t \wedge \tau \wedge \tau_n} \big)$, we have
\begin{align}\label{eq:Ito}
&\Psi \big(W^\pi_{t \wedge \tau \wedge \tau_n} \big) = \Psi(w) +\int_0^{t \wedge \tau \wedge \tau_n}  \left( \big(r W^\pi_s + (\mu-r) \pi_s - c \big) \, \Psi_w(W^\pi_s) + \dfrac{1}{2} \,\sig^2 \pi_s^2 \, \Psi_{ww}(W^\pi_s) \right) \dd s \\
& \quad + \int_0^{t \wedge \tau \wedge \tau_n} \Psi_w(W^\pi_s) \, \sig \pi_s \, \dd B_s + \int_0^{t \wedge \tau \wedge \tau_n} \big( 0 - \Psi(W^\pi_s) \big) \dd N^d_s + \int_0^{t \wedge \tau \wedge \tau_n} \big( 1 - \Psi(W^\pi_s) \big) {\bf 1}_{\{W^\pi_s < 0\}} \, \dd N^\rho_s.\notag
\end{align}
The definition of $\tau_n$ implies that the expectation of the second integral is 0.  Now, $\{N^d_t - \la t\}_{t \ge 0}$ and $\{N^\rho_t - \rho t\}_{t \ge 0}$ are two $\mathbb{H}^\pi$-martingales.  Thus, because $\Psi$ is bounded, the expectations of the third and fourth integrals equal, respectively,
\[
- \la \E^w \left[ \int_0^{t \wedge \tau \wedge \tau_n} \Psi(W^\pi_s) \, \dd s \right],
\]
and
\[
- \rho \E^w \left[ \int_0^{t \wedge \tau \wedge \tau_n} \big( \Psi(W^\pi_s) - 1 \big) {\bf 1}_{\{W^\pi_s < 0\}} \, \dd s \right].
\]
Thus, we have
\begin{align}\label{eq:3.9}
\E^w \left[ \Psi \big(W^\pi_{t \wedge \tau \wedge \tau_n} \big) \right] &= \Psi(w) - \la \E^w \left[ \int_0^{t \wedge \tau \wedge \tau_n} \Psi(W^\pi_s) \, \dd s \right] - \rho \E^w \left[ \int_0^{t \wedge \tau \wedge \tau_n} \big( \Psi(W^\pi_s) - 1 \big) {\bf 1}_{\{W^\pi_s < 0\}} \, \dd s \right] \notag \\
&\quad + \E^w \left[ \int_0^{t \wedge \tau \wedge \tau_n} \left( \big(rW^\pi_s + (\mu-r) \pi_s - c \big) \, \Psi_w(W^\pi_s) + \frac{1}{2} \,\sig^2 \pi_s^2 \, \Psi_{ww}(W^\pi_s) \right) \dd s \right] \notag \\
&\ge \Psi(w),
\end{align}
in which the inequality follows from condition 4 of the theorem.  Because $\Psi$ is bounded, it follows from the dominated convergence theorem that
\begin{equation}\label{eq:sub-mart}
\E^w \left[ \Psi \big(W^\pi_{t \wedge \tau} \big) \right] \ge \Psi(w).
\end{equation}
The same argument that gives us inequality \eqref{eq:sub-mart} shows that $\left\{ \Psi \big(W^\pi_{t \wedge \tau} \big) \right\}_{t \ge 0}$ is a $\mathbb{H}^\pi$-sub-martingale for any admissible strategy $\pi$.

From $\Psi(W^\pi_{\tau_d}) = \Psi(\De) = 0$, $\Psi(W^\pi_{\kap^\pi}) = \Psi(\K) = 1$, $\Psi(W^\pi_{\tau_{c/r}}) = \Psi(c/r) = 0$, and \hfill \break $\Psi(W^\pi_{\tau_{L}}) = \Psi(-L) = \rho/(\la + \rho)$, it follows that (dropping the superscript $\pi$ from $\kap$ for simplicity)
\begin{align}\label{eq:m}
\Psi \big(W^\pi_\tau \big) &= \Psi(W^\pi_{\tau_d}) {\bf 1}_{\{ \tau_d < (\kap \wedge \tau_{c/r} \wedge \tau_L) \}} + \Psi(W^\pi_\kap) {\bf 1}_{\{ \kap < (\tau_d \wedge \tau_{c/r} \wedge \tau_L) \}} \notag \\
&\quad + \Psi(W^\pi_{\tau_{c/r}}) {\bf 1}_{\{ \tau_{c/r} < (\tau_d \wedge \kap \wedge \tau_L) \}} + \Psi(W^\pi_{\tau_L}) {\bf 1}_{\{ \tau_L < (\tau_d \wedge \kap \wedge \tau_{c/r}) \}}  \notag \\
&= {\bf 1}_{\{ \kap < (\tau_d \wedge \tau_{c/r} \wedge \tau_L) \}} + \dfrac{\rho}{\la + \rho} \, {\bf 1}_{\{ \tau_L < (\tau_d \wedge \kap \wedge \tau_{c/r}) \}}  \notag \\
&= {\bf 1}_{\{ \kap < (\tau_d \wedge \tau_L) \}} + \dfrac{\rho}{\la + \rho} \, {\bf 1}_{\{ \tau_L < (\tau_d \wedge \kap) \}}  \notag \\
&= {\bf 1}_{\{\kap < \tau_d \}} {\bf 1}_{\{ Z^\pi_{\kap \wedge \tau_d} > -L \}} + \dfrac{\rho}{\la + \rho} \, {\bf 1}_{\{ Z^\pi_{\kap \wedge \tau_d} \le -L \}},
\end{align}
in which the third line follows because if wealth reaches $c/r$ at, say, time $t$, then $W^\pi_s = c/r$ for all $s \ge t$; recall that $\pi_s = 0$ for all $s \ge t$.  By taking the expectation of both sides of \eqref{eq:m}, we obtain
\begin{equation}\label{eq:3.12}
\E^w \left[{\bf 1}_{\{\kap < \tau_d \}} {\bf 1}_{\{ Z^\pi_{\kap \wedge \tau_d} > -L \}} + \dfrac{\rho}{\la + \rho} \, {\bf 1}_{\{ Z^\pi_{\kap \wedge \tau_d} \le -L \}} \right] = \E^w \left[ \Psi \big(W^\pi_\tau \big) \right] \ge \Psi(w).
\end{equation}
The inequality in \eqref{eq:3.12} follows from an application of the optional sampling theorem because \break $\{ \Psi(W^\pi_{t \wedge \tau}) \}_{t \ge 0}$ is a sub-martingale and $\sup_{t \ge 0} \E^w \big[\Psi(W^\pi_{t \wedge \tau}) \big] < \infty$.  By taking the infimum in \eqref{eq:3.12} over all admissible investment strategies, we obtain
\begin{equation}\label{eq:ineq}
\inf_\pi \E^w \left[ {\bf 1}_{\{\kap < \tau_d \}} {\bf 1}_{\{ Z^\pi_{\kap \wedge \tau_d} > -L \}} + \dfrac{\rho}{\la + \rho} \, {\bf 1}_{\{ Z^\pi_{\kap \wedge \tau_d} \le -L \}} \right] \ge \Psi(w).
\end{equation}
Now, by \eqref{eq:psi}, the left-hand side of \eqref{eq:ineq} equals $\psi(w)$; thus, $\psi(w) \ge \Psi(w)$.

If the individual follows a strategy $\pv$ that minimizes the right-hand side of \eqref{eq:HJB}, then \eqref{eq:3.9} is satisfied with equality, and applying the dominated convergence theorem yields
\begin{equation}
\E^w \Big[ \Psi \big(W^{\pv}_{t \wedge \tau} \big) \Big] = \Psi(w),
\end{equation}
which implies that $\{\Psi(W^{\pv}_{t \wedge \tau}) \}_{t \ge 0}$ is a martingale.  By following the same line of argument as above, we obtain
\begin{equation}
\psi(w) = \Psi(w),
\end{equation}
which demonstrates that \eqref{eq:verif} holds and that $\pv$ in \eqref{eq:pi} is an optimal investment strategy for wealth lying in $]-L, c/r[$.
\end{proof}

We end this section with a brief presentation of the related problem analyzed by Bayraktar and Young \cite{8.}, so that the reader can compare the two corresponding HJB equations.  Define the process $A^\pi = \{A^\pi_t\}_{t \ge 0}$ by
\begin{equation}\label{eq:A}
A^\pi_t = A^\pi_0 + \int_0^t {\bf 1}_{\{W^\pi_s < 0\}} \, \dd s, \quad A^\pi_0 = a \ge 0,
\end{equation}
and define the value function $M$ by
\begin{equation}\label{eq:M-L}
M(w, a) = \inf_\pi \E^{w, a} \left[ A^\pi_{\tau_d} \, {\bf 1}_{\{Z^\pi_{\tau_d} > -L \}} + \left( A^\pi_{\tau_L} + \frac{1}{\la} \right)  {\bf 1}_{\{Z^\pi_{\tau_d} \le -L \}} \right],
\end{equation}
in which we take the infimum over admissible investment strategies.  For large values of $L$, $M$ approximates the minimum expected occupation time of wealth below $0$.  Bayraktar and Young \cite{8.} showed that $M$ is given by
\[
M(w, a) = m(w) + a,
\]
in which $m$ is the unique classical solution of the following BVP on $[-L, c/r]$:
\begin{equation}\label{eq:HJB_m}
\begin{cases}
\la m(w) = {\bf 1}_{\{w < 0\}} + (rw - c) m_w(w) + \inf_{\pi} \left[(\mu-r)\pi m_w(w) + \dfrac{1}{2}\sig^2 \pi^{2} m_{ww}(w) \right], \\
m(-L) = 1/\la, \qquad m(c/r) = 0.
\end{cases}
\end{equation}
The function $m$ appears later in this paper in Section \ref{sec:ext_2}; specifically, we show that for small values of $\rho$, $\psi(w) = \rho m(w) + O(\rho^2)$ uniformly on $[-L, c/r]$.

{\bf \subsection{Solving for $\psi$ via a Related Free-Boundary Problem}}\label{sec:FBP}

In this section, we introduce a free-boundary problem (FBP) whose concave solution is the dual of $\psi$ by the Legendre transform.  To obtain the FBP, we begin with the boundary-value problem (BVP) as stated in the verification theorem, Theorem \ref{thm:verif}.
\begin{equation}\label{eq:psi_BVP}
\begin{cases}
\la \Psi + \rho( \Psi - 1) {\bf 1}_{\{w < 0\}} = (rw - c) \Psi_w + \inf \limits_\pi \left[(\mu-r)\pi \Psi_w + \dfrac{1}{2} \, \sig^2 \pi^2 \Psi_{ww} \right] , \\
\Psi(-L) = \dfrac{\rho}{\la + \rho} \, , \quad \Psi(c/r) = 0.
\end{cases}
\end{equation}
We hypothesize that the solution of this BVP is convex; thus, we can define its concave Legendre transform $\hpsi$ as follows: define $y = - \Psi_w$ and $\hpsi(y) = \Psi(w) + wy$.  Furthermore, based on related work, we suppose $\Psi_w(c/r) = 0$.  Finally, we define $\yo = - \Psi_w(0)$ and $\yL = - \Psi_w(-L)$.  We deduce the following linear FBP for $\hpsi$ for $0 \le y \le \yL$:
\begin{equation}\label{eq:hpsi_FBP}
\begin{cases}
\big( \la + \rho {\bf 1}_{\{ y > \yo \}} \big) \hpsi = -\big(r - \la - \rho {\bf 1}_{\{ y > \yo \}} \big) y \hpsi_y + \del y^2 \hpsi_{yy} + c y + \rho {\bf 1}_{\{ y > \yo \}}, \vspace{0.5ex} \\
\hpsi(0) = 0 \; \hbox{ and }  \hpsi_y(\yo) = 0, \vspace{0.5ex} \\
\hpsi(\yL) = \dfrac{\rho}{\la + \rho} - L \yL \; \hbox{ and } \hpsi_y(\yL) = - L,
\end{cases}
\end{equation}
in which $\del$ equals
\begin{equation}\label{eq:del}
\del = \dfrac{1}{2} \left( \dfrac{\mu - r}{\sig} \right)^2.
\end{equation}
We next solve this FBP; after that, we show that the convex dual of the solution of the FBP, indeed, equals $\psi$.

To solve the FBP in \eqref{eq:hpsi_FBP}, we consider the problem on the two intervals:  (1) $0 \le y \le \yo$, and (2) $\yo < y \le \yL$.  After solving the FBP on each interval separately, we impose value matching and smooth pasting at $y = \yo$ to determine $\yo$.

First, suppose $0 \le y \le \yo$; on this interval, $\hpsi$ solves
\begin{equation}\label{eq:hpsi_FBP1}
\begin{cases}
\la \hpsi = - (r - \la) y \hpsi_y + \del y^2 \hpsi_{yy} + c y, \vspace{0.5ex} \\
\hpsi(0) = 0 \; \hbox{ and }  \hpsi_y(\yo) = 0.
\end{cases}
\end{equation}
The general solution of the differential equation in \eqref{eq:hpsi_FBP1} is of the form
\begin{equation}\label{eq:hpsi_sol1}
\hpsi(y) = D_1 y^{B_1} + D_2 y^{B_2} + \dfrac{c}{r} \, y,
\end{equation}
in which $D_1$ and $D_2$ are constants to be determined, and $B_1$ and $B_2$ are given by
\begin{equation}\label{eq:B1}
B_1 = \dfrac{1}{2 \del} \left[ (r - \la + \del) + \sqrt{(r - \la + \del)^2 + 4 \del \la} \, \right] > 1,
\end{equation}
and
\begin{equation}\label{eq:B2}
B_2 = \dfrac{1}{2 \del} \left[ (r - \la + \del) - \sqrt{(r - \la + \del)^2 + 4 \del \la} \, \right] < 0.
\end{equation}

The boundary condition at $y = 0$ implies that $D_2 = 0$.  The boundary condition at $y = \yo$ implies that
\begin{equation}
0 = \hpsi_y(\yo) = D_1 B_1 \yo^{B_1 - 1} + \dfrac{c}{r},
\end{equation}
which gives us a relationship between $D_1$ and $\yo$, specifically,
\begin{equation}\label{eq:D1y0}
D_1 = - \, \dfrac{c}{r B_1} \, \yo^{1 - B_1}.
\end{equation}
We have, therefore, shown that on $[0, \yo]$, $\hpsi$ is given by
\begin{equation}\label{eq:hatm1}
\hpsi(y) =  \dfrac{c}{r}  \left[ y - \dfrac{\yo}{B_1} \left( \dfrac{y}{\yo} \right)^{B_1} \right],
\end{equation}
with $\yo$ to be determined.

Next, consider $\yo < y \le \yL$; on this interval, $\hpsi$ solves
\begin{equation}\label{eq:hpsi_FBP2}
\begin{cases}
(\la + \rho) \hpsi = - (r - \la - \rho) y \hpsi_y + \del y^2 \hpsi_{yy} + c y + \rho, \vspace{1ex} \\
\hpsi(\yL) = \dfrac{\rho}{\la + \rho} - L \yL \; \hbox{ and } \hpsi_y(\yL) = - L, \vspace{1.5ex} \\
\hpsi(\yo+) = \dfrac{c}{r} \, \dfrac{B_1 - 1}{B_1} \, \yo \; \hbox{ and } \hpsi_y(\yo+) = 0.
\end{cases}
\end{equation}
The general solution of the differential equation in \eqref{eq:hpsi_FBP2} is of the form
\begin{equation}\label{eq:hpsi_sol2}
\hpsi(y) = D_3 y^{B_3} + D_4 y^{B_4} + \dfrac{c}{r} \, y + \dfrac{\rho}{\la + \rho} \, ,
\end{equation}
in which $D_3$ and $D_4$ are constants to be determined, and $B_3$ and $B_4$ are given by
\begin{equation}\label{eq:B3}
B_3 = \dfrac{1}{2 \del} \left[ (r - \la - \rho + \del) + \sqrt{(r - \la - \rho + \del)^2 + 4 \del (\la + \rho)} \, \right],
\end{equation}
with $B_1 > B_3 > 1$, and
\begin{equation}\label{eq:B4}
B_4 = \dfrac{1}{2 \del} \left[ (r - \la - \rho + \del) - \sqrt{(r - \la - \rho + \del)^2 + 4 \del (\la + \rho)} \, \right],
\end{equation}
with $0 > B_2 > B_4$.
\\ \indent
When we use the expression for $\hpsi$ in \eqref{eq:hpsi_sol2}, and the four free-boundary conditions in \eqref{eq:hpsi_FBP2}, we obtain four equations for the four unknowns $D_3$, $D_4$, $\yo$, and $\yL$.  These four equations are as follows:
\begin{equation}\label{eq:eq1}
D_3 \yo^{B_3} + D_4 \yo^{B_4} + \dfrac{c}{r} \, \yo + \dfrac{\rho}{\la + \rho} = \dfrac{c}{r} \, \dfrac{B_1 - 1}{B_1} \, \yo ,
\end{equation}
\begin{equation}\label{eq:eq2}
D_3 B_3 \yo^{B_3} + D_4 B_4 \yo^{B_4} + \dfrac{c}{r} \, \yo = 0 ,
\end{equation}
\begin{equation}\label{eq:eq3}
D_3 \yL^{B_3} + D_4 \yL^{B_4} + \dfrac{c}{r} \, \yL + \dfrac{\rho}{\la + \rho} = \dfrac{\rho}{\la + \rho}  - L \yL ,
\end{equation}
and
\begin{equation}\label{eq:eq4}
D_3 B_3 \yL^{B_3} + D_4 B_4 \yL^{B_4} + \dfrac{c}{r} \, \yL = - L \yL .
\end{equation}
Equations \eqref{eq:eq1} and \eqref{eq:eq2} imply the following expressions for $D_3$ and $D_4$ in terms of $\yo$:
\begin{equation}\label{eq:D3_1}
D_3 = - \, \dfrac{B_1 - B_4}{B_1(B_3 - B_4)} \, \dfrac{c}{r} \, \yo^{1 - B_3} + \dfrac{B_4}{B_3 - B_4} \, \dfrac{\rho}{\la + \rho} \, \yo^{- B_3},
\end{equation}
and
\begin{equation}\label{eq:D4_1}
D_4 = \dfrac{B_1 - B_3}{B_1(B_3 - B_4)} \, \dfrac{c}{r} \, \yo^{1 - B_4} - \dfrac{B_3}{B_3 - B_4} \, \dfrac{\rho}{\la + \rho} \, \yo^{- B_4}.
\end{equation}
Similarly, \eqref{eq:eq3} and \eqref{eq:eq4} imply the following expression for $D_3$ and $D_4$ in terms of $\yL$:
\begin{equation}\label{eq:D3_2}
D_3 = - \, \dfrac{1 - B_4}{B_3 - B_4} \left( \dfrac{c}{r} + L \right) \yL^{1 - B_3} < 0,
\end{equation}
and
\begin{equation}\label{eq:D4_2}
D_4 = - \, \dfrac{B_3 - 1}{B_3 - B_4} \left( \dfrac{c}{r} + L \right) \yL^{1 - B_4} < 0.
\end{equation}
If we equate the expressions for $D_3$ in \eqref{eq:D3_1} and \eqref{eq:D3_2} and isolate $\yo$, we obtain
\begin{equation}\label{eq:yo_1}
\dfrac{\rho}{\la + \rho} \, \dfrac{1}{\yo} = - \, \dfrac{1 - B_4}{B_4} \left( \dfrac{c}{r} + L \right) \left( \dfrac{\yo}{\yL} \right)^{B_3 - 1} + \dfrac{B_1 - B_4}{B_1 B_4} \, \dfrac{c}{r} \, .
\end{equation}
Similarly, if we equate the expressions for $D_4$ in \eqref{eq:D4_1} and \eqref{eq:D4_2} and isolate $\yo$, we obtain
\begin{equation}\label{eq:yo_2}
\dfrac{\rho}{\la + \rho} \, \dfrac{1}{\yo} =  \dfrac{B_3 - 1}{B_3} \left( \dfrac{c}{r} + L \right) \left( \dfrac{\yo}{\yL} \right)^{B_4 - 1} + \dfrac{B_1 - B_3}{B_1 B_3} \, \dfrac{c}{r} \, .
\end{equation}
Finally, if we equate the expressions for $\yo$ in \eqref{eq:yo_1} and \eqref{eq:yo_2}, we deduce that $\yoL = \yo/\yL$ is a zero of $g$ in
$]0, 1[\,$, in which $g$ is defined by
\begin{equation}\label{eq:g}
g(z) = \left( \dfrac{c}{r} + L \right) \left\{ \dfrac{B_3(1 - B_4)}{B_3 - B_4} \, z^{B_3 - 1} + \dfrac{B_4(B_3 - 1)}{B_3 - B_4} \, z^{B_4 - 1} \right\} - \dfrac{c}{r} \, .
\end{equation}
It remains to show that $g$ has a unique zero in $]0, 1[$ and that $\hpsi$ is concave; we demonstrate both in the proof of the following proposition.

\begin{prop}\label{prop:hpsi_FBP}
The solution $\hpsi$ of the free-boundary problem in \eqref{eq:hpsi_FBP} is given by
\begin{equation}\label{eq:hpsi}
\hpsi(y) =
\begin{cases}
\dfrac{c}{r} \left[ y - \dfrac{\yo}{B_1} \left( \dfrac{y}{\yo} \right)^{B_1} \right], &0 \le y \le \yo, \vspace{1.5ex} \\
\dfrac{\rho}{\la + \rho} - \left\{ \left( \dfrac{c}{r} + L \right) \, \yL \left[ \dfrac{1 - B_4}{B_3 - B_4} \left( \dfrac{y}{\yL} \right)^{B_3}  + \dfrac{B_3 - 1}{B_3 - B_4}  \left( \dfrac{y}{\yL} \right)^{B_4} \right] -  \dfrac{c}{r} \, y \right\}, &\yo < y \le \yL.
\end{cases}
\end{equation}
The ratio of the free boundaries $\yoL = \yo/\yL$ is the unique zero of $g$ in $]0, 1[ \, $, in which $g$ is defined by \eqref{eq:g}.  Then, as a function of $\yoL$, $\yo$ is given by \eqref{eq:yo_1} or \eqref{eq:yo_2}, and $\yL = \yo/\yoL$.

Moreover, $\hpsi$ in \eqref{eq:hpsi} is twice continuously differentiable and strictly concave on $]0, \yL[$, except at $y = y_0$, where it is continuously differentiable with left- and right-second derivatives.
\end{prop}

\begin{proof}
By construction $\hpsi$ in \eqref{eq:hpsi} solves the FBP in \eqref{eq:hpsi_FBP} and is twice continuously differentiable, except at $y = y_0$, where it is continuously differentiable with left- and right-second derivatives.

Next, we show that $g$ has a unique zero in $]0, 1[ \,$.  Note that $\lim_{z \downarrow 0} g(z) = - \infty$ and $g(1) = L > 0$; thus, it is enough to show that $g$ increases on $]0, 1[ \,$.  Differentiate $g$ to obtain
\begin{equation}\label{eq:gz}
g_z(z) = \left( \dfrac{c}{r} + L \right) \left\{ \dfrac{B_3(B_3 - 1)(1 - B_4)}{B_3 - B_4} \, z^{B_3 - 2} - \dfrac{B_4(1 - B_4)(B_3 - 1)}{B_3 - B_4} \, z^{B_4 - 2} \right\} ,
\end{equation}
which is clearly positive because $B_3 > 1$ and $B_4 < 0$; thus, $g$ increases on $]0, 1[$ and has a unique zero in $]0, 1[ \, $.

It remains to show that $\hpsi$ is concave on $[0, \yL]$.  Observe that
\begin{equation}\label{eq:hpsi_yy}
\hpsi_{yy}(y) =
\begin{cases}
- \, \dfrac{c}{r} \, \dfrac{B_1 - 1}{\yo} \left( \dfrac{y}{\yo} \right)^{B_1 - 2}, &\quad 0 < y \le \yo, \vspace{1ex} \\
- \left( \dfrac{c}{r} + L \right) \, \dfrac{1}{\yL} \, \dfrac{(B_3 - 1)(1 - B_4)}{B_3 - B_4} \left[ B_3 \left( \dfrac{y}{\yL} \right)^{B_3 - 2}  - B_4 \left( \dfrac{y}{\yL} \right)^{B_4 - 2} \right], &\quad \yo < y \le \yL ,
\end{cases}
\end{equation}
from which we see that $\hpsi_{yy} < 0$; thus, $\hpsi$ is strictly concave.
\end{proof}

We, next, observe that the Legendre transform (see, for example, \cite{17.}) of the solution of the FBP \eqref{eq:hpsi_FBP} is, in fact, the minimum probability of lifetime exponential Parisian ruin $\psi$ for $w \in [-L, c/r]$.  Recall from Proposition \ref{prop:hpsi_FBP} that $\hpsi$ is concave on $[0, \yL]$; thus, we can define its convex dual via the Legendre transform.  Specifically, for $w \in [-L, c/r]$, define
\begin{equation}\label{eq:Leg}
\Psi(w) = \max_{y \ge 0} \left( \hpsi(y) - wy \right).
\end{equation}
By construction, $\Psi$ satisfies the conditions of Theorem \ref{thm:verif} (see the proof of Theorem \ref{thm:psi} below); thus, $\Psi$ equals $\psi$, the minimum probability of lifetime exponential Parisian ruin.

The following theorem gives explicit expressions for $\psi$ and the corresponding optimal investment strategy.

\begin{thm}\label{thm:psi}
Let $\Psi$ be defined by \eqref{eq:Leg}, in which $\hpsi$ is given in \eqref{eq:hpsi} in Proposition $\ref{prop:hpsi_FBP}$.  Then, the minimum probability of lifetime exponential Parisian ruin $\psi$ defined in \eqref{eq:psi} equals the decreasing, convex function $\Psi$ on $[-L, c/r]$.  Specifically, $\psi$ is given by
\begin{equation}\label{eq:psi_sol}
\psi(w) =
\begin{cases}
\dfrac{\rho}{\la + \rho} - \left( \dfrac{c}{r} + L \right) \, \yL \, \dfrac{(B_3 - 1)(1 - B_4)}{B_3 - B_4} \left[ \left( \dfrac{y}{\yL} \right)^{B_4}  - \left( \dfrac{y}{\yL} \right)^{B_3}  \right] , &\quad -L \le w < 0, \vspace{1.5ex} \\
\dfrac{c \yo}{rq} \left( 1 - \dfrac{rw}{c} \right)^q , &\qquad 0 \le w \le \dfrac{c}{r} \, ,
\end{cases}
\end{equation}
in which $y \in \; ]\yo, \yL]$ in the first line of \eqref{eq:psi_sol} uniquely solves
\begin{equation}\label{eq:yw}
w = - \left( \dfrac{c}{r} + L \right) \left[ \dfrac{B_3(1 - B_4)}{B_3 - B_4} \left( \dfrac{y}{\yL} \right)^{B_3 - 1}  + \dfrac{B_4(B_3 - 1)}{B_3 - B_4}  \left( \dfrac{y}{\yL} \right)^{B_4 - 1} \right]  + \dfrac{c}{r} \, ,
\end{equation}
and $q$ in the second line equals
\begin{equation}\label{eq:q}
q = \dfrac{B_1}{B_1 - 1} = \dfrac{1}{2r} \, \Big[ (r + \la + \del) + \sqrt{(r + \la + \del)^2 - 4 r \la} \, \Big] > 1.
\end{equation}
Here, $D_3$, $D_4$, $\yo$, and $\yL$ are as stated in Proposition $\ref{prop:hpsi_FBP}$. The optimal investment strategy $\pv$ is given in feedback form by \eqref{eq:pi}.  By slightly abusing notation, we write $\pv_t = \pv(W^*_t)$, in which $W^*$ is optimally controlled wealth and
\begin{equation}\label{eq:pi_sol}
\pv(w) =
\begin{cases}
\dfrac{\mu - r}{\sig^2} \left( \dfrac{c}{r} + L \right) \dfrac{(B_3 - 1)(1 - B_4)}{B_3 - B_4} \left[ B_3 \left( \dfrac{y}{\yL} \right)^{B_3 - 1}  - B_4 \left( \dfrac{y}{\yL} \right)^{B_4 - 1} \right] , &\quad -L < w < 0,  \vspace{1.5ex} \\
\dfrac{\mu - r}{\sig^2} \, \dfrac{c/r - w}{q - 1} \, , &\qquad 0 < w \le \dfrac{c}{r} \, .
\end{cases}
\end{equation}
\end{thm}

\begin{proof}
To prove this theorem, it is enough to show that $\Psi$ defined by \eqref{eq:Leg} satisfies the four conditions in Theorem \ref{thm:verif}.   First, note that the optimal value of $y$ in \eqref{eq:Leg} solves $\hpsi_y(y) = w$, and a unique such $y \in [\yo, \yL]$ exists because Proposition \ref{prop:hpsi_FBP} shows that $\hpsi_y$ is strictly decreasing.  Because of the simple form of $\hpsi$ for $y \in [0, \yo]$, we can solve for $\Psi$ explicitly for $w \in [0, c/r]$, as displayed in \eqref{eq:psi_sol}.  The relationship $\hpsi_y(y) = w$ gives us \eqref{eq:yw} for $-L \le w < 0$, and $\Psi(w) = \hpsi(y) - y \hpsi_y(y)$ gives us the first expression in \eqref{eq:psi_sol}.

Also, if we insert $\hpsi(y) = \Psi(w) - w \Psi_w(w)$, $\hpsi_y(y) = w$, and $\hpsi_{yy}(y) = -1/\Psi_{ww}(w)$ into the differential equation in \eqref{eq:hpsi_FBP}, we obtain the HJB equation in \eqref{eq:HJB}, modulo applying the first-order necessary condition for $\pi$.  Thus, $\Psi$ satisfies condition 4 of Theorem \ref{thm:verif}.

Note that $w = c/r$ corresponds to $y = 0$, and $w = -L$ corresponds to $y = \yL$.  Condition 3 of Theorem \ref{thm:verif} follows from $\Psi(c/r) = \hpsi(0) - 0 \cdot \hpsi_y(0) = 0$.  Similarly, condition 2 follows from $\Psi(-L) = \hpsi(\yL) - \yL \hpsi_y(\yL) = \rho/(\la + \rho)$.

Finally, $\Psi$ is decreasing and convex on $[-L, c/r]$ because $\hpsi$ is increasing and concave on $[0, \yL]$.  $\Psi$ is twice continuously differentiable, except at $w = 0$ where it is $\C^1$ and has left- and right-second derivatives because $\hpsi$ satisfies the same properties with the exception point equal to $y = \yo$, which corresponds to $w = 0$.  Thus, $\Psi$ satisfies condition 1 of Theorem \ref{thm:verif}.
\end{proof}

\section{Properties of the Minimum Probability of Lifetime Exponential Parisian Ruin and Optimal Investment Strategy}\label{sec:prop}

In this section, we address the following questions.

\begin{enumerate}
\item{} How does the optimal investment strategy $\pv$ given in \eqref{eq:pi_sol} compare with the optimal investment strategy when minimizing the probability of lifetime ruin, as computed in Young \cite{2.}?

\item{} How does $\pv$ compare with the optimal investment strategy $\pi_L$ when minimizing expected occupation time, as computed in Bayraktar and Young \cite{8.}?  More generally, because $\lim_{\rho \downarrow 0} \pv = \pi_L$, how does $\pv$ change as the hazard rate $\rho$ changes?

\item{} How does $\pv$ vary with respect to $w$?

\item{} How do $\pv$ and $\psi$ change as $L \to \infty$?
\end{enumerate}

Let $\pi_0$ denote the optimal investment strategy when minimizing the probability of lifetime ruin.  Recall that this is the optimal investment strategy corresponding to any ruin level; that is, it is independent of the ruin level.  Then, we have the following proposition.

\begin{prop}\label{prop:pi_pi0}
$\pv(w) = \pi_0(w)$ for $0 <  w < c/r$, and $\pv(w) > \pi_0(w)$ for $-L < w < 0$.
\end{prop}

\begin{proof}
For $0 < w < c/r$, the second part of \eqref{eq:pi_sol} shows us that $\pv(w) = \frac{\mu - r}{\sig^2} \, \frac{c/r - w}{q - 1}$, which equals $\pi_0(w)$; see Young \cite{2.}.

To prove the second part of this proposition, for $-L < w < 0$, the first part of \eqref{eq:pi_sol} implies that $\pv(w) > \pi_0(w)$ if and only if
\[
\left( \dfrac{c}{r} + L \right) \dfrac{(B_3 - 1)(1 - B_4)}{B_3 - B_4} \left[ B_3 \left( \dfrac{y}{\yL} \right)^{B_3 - 1}  - B_4 \left( \dfrac{y}{\yL} \right)^{B_4 - 1} \right] > \left( \dfrac{c}{r} - w \right) (B_1 - 1),
\]
in which $w$ and $y$ are related by \eqref{eq:yw}.  If we substitute for $w$ in the above inequality in terms of $y$ and simplify the result, we obtain the following equivalent inequality:
\[
- B_3 (1 - B_4)(B_1 - B_3) \left( \dfrac{y}{\yL} \right)^{B_3 - 1} - B_4 (B_3 - 1)(B_1 - B_4) \left( \dfrac{y}{\yL} \right)^{B_4 - 1} > 0.
\]
Because the left-hand side of this inequality decreases with respect to $y$, it is enough to show it at $y = \yL$, that is,
\[
- B_3 (1 - B_4)(B_1 - B_3) - B_4 (B_3 - 1)(B_1 - B_4) > 0,
\]
or equivalently,
\[
r(q - 1) > \del,
\]
which is straightforward to demonstrate. Thus, $\pv(w) > \pi_0(w)$ for $-L < w < 0$.
\end{proof}

Proposition \ref{prop:pi_pi0} implies that, if one seeks to minimize the probability of lifetime exponential Parisian ruin, then leveraging is {\it worse} for negative wealth than when minimizing the probability of lifetime ruin.  In Section \ref{sec:ext_1}, we compute the minimum probability of lifetime exponential Parisian ruin when $L = \infty$ and when we restrict admissible investment strategies so that the amount invested in the risky asset lies in the interval $[0, \pi_0(w)]$ when wealth equals $w \in \,]-\infty, c/r[\, $.

When minimizing expected occupation time of the individual's wealth process below zero, Bayraktar and Young \cite{8.} found a similar result.  In fact, because $B_3 \big |_{\rho = 0} = B_1$ and $B_4 \big|_{\rho = 0} = B_2$, the limit of the optimal investment strategy $\pv$ as $\rho \downarrow 0$ in \eqref{eq:pi_sol} is {\it identical} to the investment strategy to minimize expected occupation time.\footnote{Note that when $\rho = 0$, the probability of lifetime exponential Parisian ruin equals $0$ for {\it any} investment strategy, from which it follows that it does not make sense to talk about the optimal investment strategy when $\rho = 0$.  However, $\lim_{\rho \downarrow 0} \pv$ exists, and it's that limit that equals the optimal investment strategy to minimize expected occupation time.  In Section \ref{sec:ext_2}, we further demonstrate the relationship between $\psi$ and the minimum expected occupation time.}  That observation leads to the question of how $\pv$ changes with respect to $\rho$, and we answer that question in the following proposition.

\begin{prop}\label{prop:pi_rho}
Suppose $0 < \rho_1 < \rho_2$, and let $\pi_i$ denote the optimal investment strategy $\pv$ when $\rho = \rho_i$ for $i = 1, 2$. Then, $\pi_1(w) = \pi_2(w)$ for $0 < w < c/r$, and $\pi_1(w) < \pi_2(w)$ for $-L < w < 0$.
\end{prop}

\begin{proof}
Because $q$ in \eqref{eq:q} is independent of $\rho$, the second expression in \eqref{eq:pi_sol} shows us that $\pv(w)$ is independent of $\rho$ for $0 < w < c/r$.

For $-L < w < 0$, we use a comparison argument to prove that $\pi_1(w) < \pi_2(w)$.  To that end, begin with the ODE for $\psi$ when $-L < w < 0$, namely,
\[
(\la + \rho)\psi - \rho = (rw - c)\psi_w - \del \, \dfrac{\psi_w^2}{\psi_{ww}}.
\]
Because $\pv(w) = - \, \frac{\mu - r}{\sig^2} \, \frac{\psi_w}{\psi_{ww}}$, we can rewrite this ODE as follows:
\[
(\la + \rho)\psi - \rho = (rw - c)\psi_w + \dfrac{\mu - r}{2} \, \pv \psi_w.
\]
Differentiate this ODE with respect to $w$ to obtain
\[
(\la + \rho)\psi_w  = r \psi_w + (rw - c)\psi_{ww} + \dfrac{\mu - r}{2} \, \pv_w \psi_w + \dfrac{\mu - r}{2} \, \pv \psi_{ww}.
\]
Divide by $\psi_w$ and again use $\pv(w) = - \, \frac{\mu - r}{\sig^2} \, \frac{\psi_w}{\psi_{ww}}$ to obtain the following first-order ODE for $\pv$ on $]-L, 0[\,$:
\begin{equation}\label{eq:pi_ode}
\pv_w = \dfrac{2}{\mu - r} \, (\la + \rho + \del - r) - \dfrac{2}{\sig^2} \, \dfrac{c - rw}{\pv} \, .
\end{equation}
Also, from Theorem \ref{thm:psi}, we have the initial condition
\begin{equation}\label{eq:pi_L}
\pv(-L+) = \dfrac{2r}{\mu - r} \left( \dfrac{c}{r} + L \right),
\end{equation}
which is independent of $\rho$.  Thus, we have $\pi_1(-L+) = \pi_2(-L+)$, and \eqref{eq:pi_ode} implies that \\
$$(\pi_1)_w(-L+) < (\pi_2)_w(-L+).$$
Define the {\it defect} ${\rm P_1}$ corresponding to $\pi_1$ on $]-L, 0[$ as in Chapter 8 of Walter \cite{18.} by
\[
{\rm P_1} f = f_w - \dfrac{2}{\mu - r} \, (\la + \rho_1 + \del - r) + \dfrac{2}{\sig^2} \, \dfrac{c - rw}{f}.
\]
Then, we have ${\rm P_1} \pi_1 = 0$ and
\begin{align*}
{\rm P_1} \pi_2 &= (\pi_2)_w - \dfrac{2}{\mu - r} \, (\la + \rho_1 + \del - r) + \dfrac{2}{\sig^2} \, \dfrac{c - rw}{\pi_2} \\
&= \left[ \dfrac{2}{\mu - r} \, (\la + \rho_2 + \del - r) - \dfrac{2}{\sig^2} \, \dfrac{c - rw}{\pi_2} \right] - \dfrac{2}{\mu - r} \, (\la + \rho_1 + \del - r) + \dfrac{2}{\sig^2} \, \dfrac{c - rw}{\pi_2} \\
&= \dfrac{2}{\mu - r} \, (\rho_2 - \rho_1) > 0.
\end{align*}
From Chapter 8 of \cite{18.}, the three conditions $\pi_1(-L+) = \pi_2(-L+)$, $(\pi_1)_w(-L+) < (\pi_2)_w(-L+)$, and ${\rm P_1} \pi_1 < {\rm P_1} \pi_2$ imply that $\pi_1(w) < \pi_2(w)$ for $-L < w < 0$.
\end{proof}

Proposition \ref{prop:pi_rho} implies that, as the exponential excursion clock's hazard rate increases, the individual invests more in the risky asset in order to get her wealth above $0$, at which point the exponential excursion clock ``turns off.''  Also, as a corollary of Proposition \ref{prop:pi_rho}, we deduce that, because $\rho > 0$, $\pv$ on $]-L, 0[$ is greater than the optimal investment strategy to minimize expected occupation time.

We also have another corollary of Proposition \ref{prop:pi_rho} that examines how $\pv$ changes with $\la$.

\begin{cor}\label{cor:pi_la}
As $\la$ increases, $\pv$ decreases on $]0, c/r[$ and increases on $]-L, 0[\, $.
\end{cor}

\begin{proof}
From Young \cite{2.}, we know that $\pi_0$ decreases with $\la$; thus, on $]0, c/r[ \,$, $\pv$ also decreases with $\la$.  From the proof of Proposition \ref{prop:pi_rho}, we apply the same argument with $\rho$ and $\la$ interchanged to deduce that $\pv$ on $]-L, 0[$ increases with $\la$.
\end{proof}

When minimizing the probability of lifetime ruin, the optimal investment strategy is a linear, decreasing function of wealth.  From Theorem \ref{thm:psi}, we know that the same is true for $\pv$ on $]0, c/r[\,$.  Therefore, we wish to see how $\pv$ varies with wealth on $]-L, 0[\, $.

\begin{prop}\label{prop:pi_cx}
On $]-L, 0[\,$, $\pv$ is a convex function of wealth.  Moreover, $\pv$ decreases with respect to wealth on $]-L, 0[$ if and only if
\begin{equation}\label{eq:pi_decr}
\left( \dfrac{c}{r} + L \right) \left( - \, \dfrac{B_4}{B_3 - 1} \right)^{\frac{B_3 - 1}{B_3 - B_4}} \left( \dfrac{B_3}{1 - B_4} \right)^{\frac{1 - B_4}{B_3 - B_4}} \, \dfrac{\la + \rho + \del - r}{\del} \le \dfrac{c}{r} \, ,
\end{equation}
which is automatic if $r \ge \la + \rho + \del$.  Finally, $\pv$ increases with respect to wealth on $]-L, 0[$ if and only if $r \le \la + \rho$.
\end{prop}

\begin{proof}
For $-L < w < 0$, from \eqref{eq:pi_sol}, we know that
\[
\pv(w) \propto B_3 z^{B_3 - 1} - B_4 z^{B_4 - 1},
\]
in which $z \in ]\yoL, 1[$ uniquely solves
\[
\dfrac{c}{r} - w = \left( \dfrac{c}{r} + L \right) \left[ \dfrac{B_3(1 - B_4)}{B_3 - B_4} \, z^{B_3 - 1} + \dfrac{B_4(B_3 - 1)}{B_3 - B_4} \, z^{B_4 - 1} \right].
\]
By differentiating $z$'s equation with respect to $w$, we obtain
\[
-1 = \left( \dfrac{c}{r} + L \right) \dfrac{(B_3 - 1)(1 - B_4)}{B_3 - B_4} \big( B_3 z^{B_3 - 2} - B_4 z^{B_4 - 2} \big) \dfrac{\dd z}{\dd w},
\]
from which we note that $z$ decreases with respect to $w$.  Thus, when we differentiate $\pv$ with respect to $w$, we get
\begin{align*}
\dfrac{\dd}{\dd w} \pv(w) &\propto \big( B_3(B_3 - 1) z^{B_3 - 2} + B_4(1 - B_4) z^{B_4 - 2} \big) \dfrac{\dd z}{\dd w} \\
&\propto - \, \dfrac{B_3(B_3 - 1) z^{B_3 - 2} + B_4(1 - B_4) z^{B_4 - 2}}{B_3 z^{B_3 - 2} - B_4 z^{B_4 - 2}} \, ,
\end{align*}
and, because $\dd z/\dd w < 0$,
\begin{align*}
\dfrac{\dd^2}{\dd w^2} \pv(w) &\propto \dfrac{\dd}{\dd z} \dfrac{B_3(B_3 - 1) z^{B_3 - 2} + B_4(1 - B_4) z^{B_4 - 2}}{B_3 z^{B_3 - 2} - B_4 z^{B_4 - 2}} \\
&\propto \big( B_3 z^{B_3 - 2} - B_4 z^{B_4 - 2} \big) \big( B_3(B_3 - 1)(B_3 - 2) z^{B_3 - 3} - B_4(1 - B_4)(2 - B_4) z^{B_4 - 3} \big) \\
&\quad - \big( B_3(B_3 - 1) z^{B_3 - 2} + B_4(1 - B_4) z^{B_4 - 2} \big)  \big( B_3(B_3 - 2) z^{B_3 - 3} + B_4(2 - B_4) z^{B_4 - 3} \big) \\
&=  -B_3 B_4 (B_3 - B_4)^2 \, z^{B_3 + B_4 - 5} > 0.
\end{align*}
Thus, $\pv$ is convex, that is, $\pv_w(w)$ increases with respect to $w$ for $-L < w < 0$.

It follows that, if $\pv_w(0-) \le 0$, then $\pv_w(w) < 0$ for all $-L < w < 0$.  From \eqref{eq:pi_ode}, we deduce that $\pv_w(0-) \le 0$ if and only if
\[
\la + \rho + \del - r \le \dfrac{\mu - r}{\sig^2} \, \dfrac{c}{\pv(0-)},
\]
or equivalently, by using the expression for $\pv$ from \eqref{eq:pi_sol} and simplifying the result,
\begin{equation}\label{eq:pi_decr_yoL}
\yoL \ge \left( - \, \dfrac{B_4(1 - B_4)}{B_3(B_3 - 1)} \right)^{\frac{1}{B_3 - B_4}}.
\end{equation}
Because $g$ in \eqref{eq:g} is identically zero for $z = \yoL$ and is increasing with respect to $z$, inequality \eqref{eq:pi_decr_yoL} holds if and only if $g$ evaluated at the right-hand side of \eqref{eq:pi_decr} is non-positive, which is equivalent to \eqref{eq:pi_decr}.

Also, because $\pv$ is convex, it follows that, if $\pv_w(-L+) \ge 0$, then $\pv_w(w) > 0$ for all $-L < w < 0$.  From \eqref{eq:pi_ode} and \eqref{eq:pi_L}, we deduce that $\pv_w(-L+) \ge 0$ if and only if $r \le \la + \rho$.
\end{proof}

In the next corollary, we consider the case for which $\la + \rho < r < \la + \rho + \del$.

\begin{cor}\label{cor:pi_mono}
Suppose $\la + \rho < r < \la + \rho + \del;$ then, there exists $L_0 > 0$ such that if $0 < L \le L_0$, then $\pv$ decreases with respect to wealth on $]-L, 0[\,$.  Furthermore, if $L > L_0$, then $\pv$ first decreases then increases with respect to wealth on $]-L, 0[\,$.
\end{cor}

\begin{proof}
To prove this corollary, it is enough to show that $f = f(r)$ decreases with respect to $r$ when $\la + \rho < r < \la + \rho + \del$, in which $f$ is defined by
\[
f(r) = \left( - \, \dfrac{B_4}{B_3 - 1} \right)^{\frac{B_3 - 1}{B_3 - B_4}} \left( \dfrac{B_3}{1 - B_4} \right)^{\frac{1 - B_4}{B_3 - B_4}},
\]
because $f(\la + \rho) = 1$.  To that end, note that
\[
\dfrac{\partial B_3}{\partial r} = \dfrac{B_3}{\sqrt{(r - \la - \rho + \del)^2 + 4 \del (\la + \rho)}} \, ,
\]
\[
\dfrac{\partial B_4}{\partial r} = - \, \dfrac{B_4}{\sqrt{(r - \la - \rho + \del)^2 + 4 \del (\la + \rho)}} \, ,
\]
and
\[
\dfrac{\partial }{\partial r} \, \dfrac{B_3 - 1}{B_3 - B_4} = \dfrac{B_3 + B_4 - 2 B_3 B_4}{(B_3 - B_4)^2 \sqrt{(r - \la - \rho + \del)^2 + 4 \del (\la + \rho)}} = - \, \dfrac{\partial }{\partial r} \, \dfrac{1 - B_4}{B_3 - B_4} \, .
\]
Differentiate $\ln f(r)$ with respect to $r$ to obtain
\begin{align*}
f'(r) &\propto \dfrac{B_3 + B_4 - 2 B_3 B_4}{B_3 - B_4} \, \ln \left( - \, \dfrac{B_4}{B_3 - 1} \right) + (B_3 - 1) \left( - 1 - \dfrac{B_3}{B_3 - 1} \right) \\
&\quad - \dfrac{B_3 + B_4 - 2 B_3 B_4}{B_3 - B_4} \, \ln \left( \dfrac{B_3}{1 - B_4} \right) + (1 - B_4) \left( 1 - \dfrac{B_4}{1 - B_4} \right) \\
&\propto \ln \left( - \, \dfrac{B_4(1 - B_4)}{B_3(B_3 - 1)} \right) - 2 \, \dfrac{(B_3 - B_4)(r - \la - \rho)}{r + \la + \rho + \del}.
\end{align*}
Because $r > \la + \rho$, $f'(r) < 0$ if
\[
- \, \dfrac{B_4(1 - B_4)}{B_3(B_3 - 1)}  < 1,
\]
which is straightforward to show when $r > \la + \rho$.  We have shown that $f$ decreases with respect to $r$, from which the statement of the corollary follows.
\end{proof}

Proposition \ref{prop:pi_cx} shows that if the rate of return on the riskless asset is low enough, that is, if $r \le \la + \rho$, then the individual will borrow more money to invest in the risky asset as wealth gets closer to zero.  Moreover, Corollary \ref{cor:pi_mono} shows that if the rate is somewhat larger but not too large, that is, if $\la + \rho < r < \la + \rho + \del$, and if $L$ is large enough, that is, if $L > L_0$, then the individual will also borrow more money for investment purposes as wealth gets closer to zero.  In other words, because the borrowing rate is low (or moderate with $L$ far away), the individual will take on more debt as wealth approaches zero in order to get wealth above zero and, thereby, get wealth out of danger from the exponential excursion clock.

In the final proposition, we examine how $\psi$ and $\pv$ change as $L$ increases. Because $\pv$ is independent of $L$ on $]0, c/r[\, $, we focus on how $\pv$ changes with $L$ on $]-L, 0[\, $.  Also, because the proof of this proposition is parallel to the proofs in Section 3.3 in Bayraktar and Young \cite{8.}, in the interest of space, we omit it.

\begin{prop}\label{prop:L}
The optimal investment strategy $\pv$ increases with increasing $L$ on $]-L, 0[\, $, and as $L \to \infty$, $\pv(w) \to \infty$ linearly with respect to $L$, for all $w < 0$.  Moreover, 
\[
\lim_{L \to \infty} \psi(w) = \inf_\pi \E^w(\kap^\pi < \tau_d) = 0,
\]
for all $w \le c/r$.  \qed
\end{prop}

As $L$ goes to $\infty$, once wealth becomes negative, the investment strategy grows linearly with respect to $L$ and becomes infinitely large, which leverages the wealth back into positive territory with probability $1$. Thus, the minimum probability of lifetime exponential Parisian ruin goes to $0$ for all $w$, which is trivially convex and non-decreasing on all of $\, ]-\infty, c/r]$.

We end this section with a numerical example, and we use the following parameter values:
\begin{itemize}
\item{} $r = 0.04$, $\mu = 0.08$, and $\sig = 0.20$, which implies that $\del = 0.02$.
\item{} $c = 1$, which implies that the safe level equals $25$.
\item{} $L = 100$.
\item{} $\la = 0.01$, which gives an expected future lifetime of 100 years.
\end{itemize}

We acknowledge that the expected future lifetime is too long to be realistic, but we wish to demonstrate the results of Proposition \ref{prop:pi_cx} by varying $\rho$ so that $r < \la + \rho$.  In Figure \ref{fig:pi}, we plot the optimal investment strategy $\pv$ for various values of $\rho$.  We also plot the optimal investment strategy $\pi_L$ to minimize expected occupation time (recall that it corresponds to
$\lim_{\rho \downarrow 0} \pv$) and the optimal investment strategy $\pi_0$ to minimize the probability of lifetime ruin.  When $\rho = 0$ or $0.01$, inequality \eqref{eq:pi_decr} holds, which implies that $\pv$ decreases with respect to wealth on $]-L, 0 [\,$, which we observe in Figure \ref{fig:pi}.  When $\rho = 0.02$, inequality \eqref{eq:pi_decr} does not hold and we have $\la + \rho < r < \la + \rho + \del$; thus, Proposition \ref{prop:pi_cx} and Corollary \ref{cor:pi_mono} imply that $\pv$ first decreases then increases with respect to wealth on $]-L, 0[\,$, and we see that in Figure \ref{fig:pi}.  When $\rho \ge 0.03$, then $r \le \la + \rho$, which implies that $\pv$ increases with respect to wealth on $]-L, 0[$ which we observe in Figure \ref{fig:pi}.  Also, note that $\pv$ increases with increasing $\rho$, which we expect from Proposition \ref{prop:pi_rho}, and $\pv > \pi_0$, which we expect from Proposition \ref{prop:pi_pi0}.

\begin{figure}[!htbp]
\centering
\includegraphics[width=4.5in]{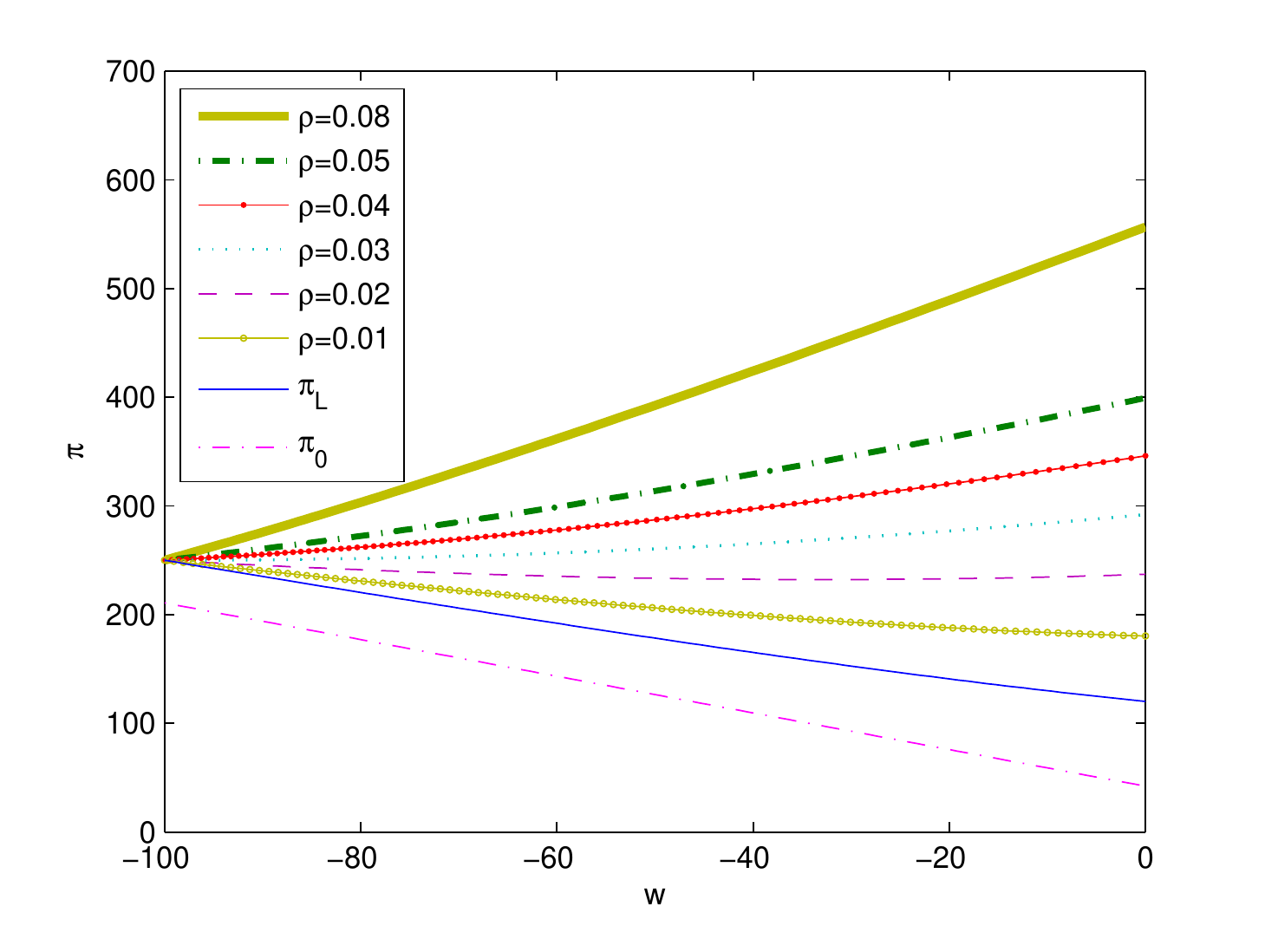}
\caption{The optimal investment strategy $\pv$ for various values of $\rho$, including $\pi_L$ and $\pi_0$.}
\label{fig:pi}
\end{figure}

\section{Restricting $\pi$ and Asymptotic Expansion in $\rho$}\label{sec:ext}

Proposition \ref{prop:pi_pi0} shows that $\pv$ equals the optimal investment strategy $\pi_0$ to minimize the probability of lifetime ruin when wealth is positive, and $\pv > \pi_0$ when wealth is negative.  Therefore, in Section \ref{sec:ext_1}, we limit the investment strategy so that the amount invested in the risky asset is no greater than when minimizing the probability of lifetime ruin, and we compute the minimum probability of lifetime exponential Parisian ruin under that restriction.

In Section 4.2, we provide an asymptotic expansion of the minimum probability of lifetime exponential Parisian ruin for small values of $\rho$, and we show that by following investment strategy corresponding to the first-order term, the value function is within $(\rho/\la)^2$ of $\psi$.

{\bf \subsection{Restricting $0 \le \pi \le \pi_0$}}\label{sec:ext_1}

Consider the probability of lifetime exponential Parisian ruin on $]-\infty, c/r[\,$, and define the corresponding value function $\psi_0$ by
\begin{equation}\label{eq:psi)}
\psi_0(w) = \inf_{\pi \in {\mathcal A}_0} \P^w \big( \kap^\pi< \tau_d \big),
\end{equation}
in which ${\mathcal A}_0$ is the set of admissible investment strategies as in Section 2, with the additional restriction that $0 \le \pi_t \le \pi_0(W_t)$ with probability 1 for all $t \ge 0$.  The optimal strategy $\pv$ is automatically positive when wealth lies below the safe level.  However, when we restrict admissible investment strategies by requiring $\pi_t \le \pi_0(W_t)$, we anticipate that $\psi_0$ will be concave on part of its domain, and we wish to avoid investment strategies that allow the individual to sell the risky asset short.  Thus, we add the restriction that $\pi_t \ge 0$ for all $t \ge 0$.

Via a verification theorem similar to Theorem \ref{thm:verif}, if we find a classical solution of the following BVP, then that solution equals $\psi_0$:
\begin{equation}\label{eq:psi0_BVP}
\begin{cases}
\la \Psi + \rho( \Psi - 1) {\bf 1}_{\{w < 0\}} = (rw - c) \Psi_w + \inf \limits_{0 \le \pi \le \pi_0(w)} \left[(\mu-r)\pi \Psi_w + \dfrac{1}{2} \, \sig^2 \pi^2 \Psi_{ww} \right] , \\
\lim \limits_{w \to - \infty} \Psi(w) = \dfrac{\rho}{\la + \rho} \, , \quad \Psi(c/r) = 0.
\end{cases}
\end{equation}
Under the ansatz that the optimal investment strategy thus restricted equals $\pi_0$, one can easily compute the solution of the BVP in \eqref{eq:psi0_BVP}.  In the following proposition, we present this solution and prove that it equals the value function $\psi_0$.

\begin{prop}\label{prop:psi0}
The minimum probability of lifetime exponential Parisian ruin $\psi_0$ under the restriction that $0 \le \pv \le \pi_0$ is given by
\begin{equation}\label{eq:psi0_sol}
\psi_0(w) =
\begin{cases}
 \dfrac{\rho}{\la + \rho} \left\{ 1 - \dfrac{q}{q + \alp} \left( 1 - \dfrac{rw}{c} \right)^{-\alp} \right\}, &\quad -\infty < w < 0, \vspace{1.5ex} \\
 \dfrac{\rho}{\la + \rho} \, \dfrac{\alp}{q + \alp}  \left( 1 - \dfrac{rw}{c} \right)^q , &\qquad 0 \le w \le \dfrac{c}{r} \, ,
\end{cases}
\end{equation}
in which $\alp$ equals
\begin{equation}\label{eq:alp}
\alp = \dfrac{q - 1}{2 \del} \left[ - (r - \la + \del) + \sqrt{(r - \la + \del)^2 + 4 \del (\la + \rho)} \, \right] > 0.
\end{equation}
The optimal investment strategy is given in feedback form via $\pi_0;$ specifically, when wealth equals $w < c/r$, the optimal amount to invest in the risky asset equals
\begin{equation}\label{eq:pi0}
\pi_0(w) =  \dfrac{\mu - r}{\sig^2} \, \dfrac{c/r - w}{q - 1} \, .
\end{equation}
\end{prop}

\begin{proof}
It is straightforward to show that the expression in \eqref{eq:psi0_sol} solves the BVP in \eqref{eq:psi0_BVP} with $\pi = \pi_0$.  It remains to show that $\pi_0$ is the optimal investment strategy, that is,
\begin{equation}\label{eq:pi0_opt}
\hbox{arg min}_{0 \le \pi \le \pi_0(w)} \left[ (\mu-r)\pi \psi_0'(w) + \dfrac{1}{2} \, \sig^2 \pi^2 \psi_0''(w) \right] = \pi_0(w).
\end{equation}
$\pi_0$ is the unrestricted minimizer of the left-hand side of equation \eqref{eq:pi0_opt} when $0 \le w < c/r$.  Next, consider the expression in square brackets in \eqref{eq:pi0_opt} for $w < 0$:
\[
(\mu-r)\pi \psi_0'(w) + \dfrac{1}{2} \, \sig^2 \pi^2 \psi_0''(w) \propto -(\mu - r) \pi \left( \dfrac{c}{r} - w \right) - \dfrac{1}{2} \, \sig^2 \pi^2 (\alp + 1).
\]
Define the parabola $f$ by
\[
f(\pi) = -(\mu - r) \pi \left( \dfrac{c}{r} - w \right) - \dfrac{1}{2} \, \sig^2 \pi^2 (\alp + 1).
\]
The graph of $f$ is a downward facing parabola with apex at a negative value of $\pi$.  Thus, the argmin of $f(\pi)$ for $\pi \in [0, \pi_0(w)]$ equals the right limit of the interval $\pi_0(w)$.
\end{proof}

{\bf \subsection{Asymptotic Expansion of $\psi$ for Small $\rho$}}\label{sec:ext_2}

In this section, we show that $\psi(w) = \rho m(w) + O(\rho^2)$ uniformly with respect to $w$ on $[-L, c/r]$, in which $m$ is the value function associated with the problem of minimizing expected occupation time modulo the current value of the time of occupation; see Bayraktar and Young \cite{8.}.  To that end, begin by rewriting the boundary-value problem in \eqref{eq:psi_BVP} as follows:
\begin{equation}\label{eq:psi_BVP2}
\begin{cases}
{\rm{F}}(w, u, u_w, u_{ww}) = 0, \\
u(-L) = \dfrac{\rho}{\la + \rho} \, , \quad u(c/r) = 0,
\end{cases}
\end{equation}
in which the operator ${\rm{F}}$ is given by
\[
{\rm{F}}(w, u, u_w, u_{ww}) = \la u + \rho (u - 1) {\bf 1}_{\{ w < 0 \}} - \inf_{\pi} \left\{\big[ rw - c + (\mu-r)\pi \big] u_w + \frac{1}{2}\sig^2\pi^2 u_{ww} \right\}.
\]

We prove the following comparison lemma for sub- and super-solutions of \eqref{eq:psi_BVP2}.  Define the set of functions $\C$ by those that are in $\C^2\big(\; ]-L, c/r[ \, \big)$, except at $0$ where they are $\C^1$ and have left- and right-second derivatives.

\begin{lemma}\label{lem:comp}
Let $u, v \in \C$ be such that ${\rm{F}}(w, u, u_w, u_{ww}) < {\rm{F}}(w, v, v_w, v_{ww})$ for all $w \in \,]-L, c/r[\,$.\footnote{At $w = 0$, we mean that this inequality holds as we take both left- and right-limits; that is, both
\[
{\rm{F}}\big(0, u(0), u_w(0), u_{ww}(0-)\big) < {\rm{F}}\big(0, v(0), v_w(0), v_{ww}(0-)\big)
\]
and
\[
{\rm{F}}\big(0, u(0), u_w(0), u_{ww}(0+)\big) < {\rm{F}}\big(0, v(0), v_w(0), v_{ww}(0+)\big)
\]
hold.} If $u(-L) \le v(-L)$ and $u(c/r) \le v(c/r)$, then $u(w) < v(w)$ for all $w \in \,]-L, c/r[\,$.
\end{lemma}

\begin{proof}
First, if the maximum of $u - v$ occurs on the boundary of $]-L, c/r[\,$, but not in the interior, then $u < v$ in the interior because $u \le v$ on the boundary by assumption.  Second, if $u - v$ attains a strictly negative maximum in the interior of $]-L, c/r[ \,$, then we also have $u < v$ in the interior.

Third, if $u - v$ attains a non-negative maximum at $w_0 \in \,]-L, c/r[\,$, then $u_w(w_0) = v_w(w_0)$ and $u_{ww}(w_0) \le v_{ww}(w_0)$.\footnote{If $w_0 = 0$, then we have both $u_{ww}(0-) \le v_{ww}(0-)$ and $u_{ww}(0+) \le v_{ww}(0+)$.}  Because ${\rm{F}}(w_0, u(w_0), u_w(w_0), u_{ww}(w_0)) < \infty$, either $u_{ww}(w_0) > 0$ or $u_{w}(w_0) = u_{ww}(w_0) = 0$.  In the former case, we have $v_{ww}(w_0) > 0$, and
\begin{align*}
0 &< {\rm{F}}(w_0, v(w_0), v_w(w_0), v_{ww}(w_0)) - {\rm{F}}(w_0, u(w_0), u_w(w_0), u_{ww}(w_0))\\
&= \left( \la v(w_0) + \rho (v(w_0) - 1) {\bf 1}_{\{ w_0 < 0 \}} \right) - \left( \la u(w_0) + \rho (u(w_0) - 1) {\bf 1}_{\{ w_0 < 0 \}} \right) + \del \, \frac{v^2_w(w_0)}{v_{ww}(w_0)} - \del \, \frac{u^2_w(w_0)}{u_{ww}(w_0)} \\
&=  \big(\la + \rho {\bf 1}_{\{ w_0 < 0 \}} \big)\big( v(w_0) - u(w_0) \big) + \del \, \frac{v^2_w(w_0)(u_{ww}(w_0)-v_{ww}(w_0))}{u_{ww}(w_0)v_{ww}(w_0)}\leq 0,
\end{align*}
a contradiction. In the latter case, we have $v_w(w_0) = 0$ and $v_{ww}(w_0) \ge 0$, from which it follows that
\begin{align*}
0 &< {\rm{F}}(w_0, v(w_0), v_w(w_0), v_{ww}(w_0)) - {\rm{F}}(w_0, u(w_0), u_w(w_0), u_{ww}(w_0))\\
&= \left( \la v(w_0) + \rho (v(w_0) - 1) {\bf 1}_{\{ w_0 < 0 \}} \right) - \left( \la u(w_0) + \rho (u(w_0) - 1) {\bf 1}_{\{ w_0 < 0 \}} \right) \\
&= \big(\la + \rho {\bf 1}_{\{ w_0 < 0 \}} \big)\big( v(w_0) - u(w_0) \big) \le 0,
\end{align*}
a contradiction.
\end{proof}

\begin{remark}
If we only want non-strict comparison, that is, $u\leq v$, then the sub-$($super-$)$solution property can be weakened to ${\rm F}(w, u, u_w, u_{ww}) \leq {\rm{F}}(w, v, v_w, v_{ww})$, with ${\rm{F}}(w, u, u_w, u_{ww})$ finite.   \qed
\end{remark}

Next, we turn to obtaining an asymptotic expression for $\psi$ for small values of $\rho$.  We focus on the case for which $w < 0$ because we have a semi-explicit expression for $\psi$ when $w \ge 0$, namely, $\psi(w) = \bet \left( 1 - rw/c \right)^q$ for some $\bet \in \;]0, 1[ \,$.  For $w < 0$, let
\begin{equation}\label{eq:expand}
\psio(w) + \rho \psil(w) + \rho^2 \psi^{(2)}(w) + \cdots
\end{equation}
be an asymptotic expansion of $\psi(w)$ as $\rho \to 0$; here, $\psio, \psil, \dots$ are independent of $\rho$.  By substituting \eqref{eq:expand} into \eqref{eq:psi_BVP} and collecting the terms of order $\rho^0$, we obtain an ODE for $\psio$.
\[
\la \psio = (rw - c) \psio_w - \del \, \frac{\Big(\psio_w \Big)^2}{\psio_{ww}} \, .
\]
Write the boundary condition at $w = -L$ in powers of $\rho$ as follows:
\begin{equation}\label{eq:bndry}
\psi(-L) = \dfrac{\rho}{\la + \rho} = \dfrac{\rho}{\la} \, \dfrac{1}{1 + \frac{\rho}{\la}} = \dfrac{\rho}{\la} - \left( \dfrac{\rho}{\la} \right)^2 + \left( \dfrac{\rho}{\la} \right)^3 - \cdots.
\end{equation}
Thus, to order $\rho^0$, the boundary condition for $\psio$ at $w = -L$ is $\psio(-L) = 0$.  This boundary condition and $\psi(c/r) = 0$ lead us to the $\rho^0$-order term of $\psio \equiv 0$.

By collecting the terms of order $\rho^1$, we obtain an ODE for $\psil$ for $w < 0$.
\begin{equation}\label{eq:psil_ODE}
\la \psil = 1 + (rw - c) \psil_w - \del \, \frac{\big(\psil_w \big)^2}{\psil_{ww}} \, ,
\end{equation}
with boundary conditions $\psil(-L) = 1/\la$ and $\psil(c/r) = 0$.  For $w \ge 0$, the ODE is as on $w < 0$ but without the $1$ on the right-hand side.  This BVP is {\it identical} to the one solved by $m$ in \eqref{eq:HJB_m}; thus, $\psil(w) = m(w)$.

In the next proposition, we show that $\rho \psil(w) = \rho m(w)$ approximates $\psi$ to order $\rho^2$ as $\rho \to 0$, uniformly for
$w \in \;]-L, c/r[\,$.

\begin{prop}\label{prop:psil}
\[
\psi(w) = \rho m(w) + O(\rho^2),
\]
as $\rho \to 0$ uniformly in $w \in \,]-L, c/r[\,$.
\end{prop}

\begin{proof}
First, because $m$ is convex, we compute
\begin{align}\label{eq:F_psil}
{\rm{F}}\big(w,  \rho m,  \rho m_w, \rho m_{ww} \big) &=  \la \rho m + \rho (\rho m - 1) {\bf 1}_{\{ w < 0 \}} \notag \\
& \quad - \inf_{\pi} \left\{\big(rw - c + (\mu-r) \pi \big) \rho m_w + \frac{1}{2}\sig^2\pi^2 \rho m_{ww} \right\} \notag \\
&= \rho \left\{ \la m + (\rho m - 1) {\bf 1}_{\{ w < 0 \}}  - (rw - c)m_w + \del \, \frac{\big(m_w \big)^2}{m_{ww}} \right\} \notag \\
&= \rho^2 m(w) {\bf 1}_{\{ w < 0 \}},
\end{align}
which, because $m \in [0, 1/\la]$, we deduce that
\begin{equation}\label{eq:F_psil_bnds}
0 \le {\rm{F}} \big(w,  \rho m,  \rho m_w, \rho m_{ww} \big) \le \dfrac{\rho^2}{\la} \, ,
\end{equation}
strictly for $-L < w < c/r$.  Also,
\[
\psi(-L) = \dfrac{\rho}{\la + \rho} < \dfrac{\rho}{\la} = \rho m(-L),
\]
and
\[
\psi(c/r) = 0 = \rho m(c/r).
\]
Thus, because ${\rm{F}}(w, \psi, \psi_w, \psi_{ww}) = 0$, Lemma \ref{lem:comp} implies that $\psi(w) < \rho m(w)$ for all $-L \le w < c/r$.

Via a calculation similar to the one in \eqref{eq:F_psil}, one can show that
\[
{\rm{F}}\big(w,  \rho m - (\rho/\la)^2,  \rho m_w, \rho m_{ww} \big) =
\begin{cases}
- \, \dfrac{\rho^2}{\la}, &\quad w \ge 0,\vspace{1ex} \\
- \rho^2 \left( \dfrac{1}{\la} - m(w) \right) - \dfrac{\rho^3}{\la^2}, &\quad w < 0,
\end{cases}
\]
which is negative because $m(w)$ decreases from $1/\la$ to $0$ as $w$ increases from $-L$ to $c/r$.  Also,
\begin{equation}\label{eq:4_13}
\rho m(-L) - \left( \dfrac{\rho}{\la} \right)^2 = \dfrac{\rho}{\la} -  \left( \dfrac{\rho}{\la} \right)^2 <  \dfrac{\rho}{\la + \rho} = \psi(-L),
\end{equation}
and
\begin{equation}\label{eq:4_14}
\rho m(c/r) - \left( \dfrac{\rho}{\la} \right)^2 = - \left( \dfrac{\rho}{\la} \right)^2 < 0 = \psi(c/r).
\end{equation}
Thus, because ${\rm F}(w, \psi, \psi_w, \psi_{ww}) = 0$, Lemma \ref{lem:comp} implies that $\rho m(w) - (\rho/\la)^2 < \psi(w)$ for all \break $-L \le w \le c/r$.

We have shown that
\[
\rho m(w) - \left( \dfrac{\rho}{\la} \right)^2 < \psi(w) < \rho m(w),
\]
for all $-L \le w < c/r$, which proves the proposition.
\end{proof}

Finally, if the individual follows the investment strategy defined by $ \rho m$, which is given by $\lim_{\rho \downarrow 0} \pv$, then the resulting value function is within order $O(\rho^2)$ of the minimum probability of lifetime exponential Parisian ruin $\psi$.

\begin{cor}\label{cor:piL}
The investment strategy given in feedback form via
\[
\pi_L(w) = - \dfrac{\mu-r}{\sigma^2} \, \dfrac{  m_w(w)}{  m_{ww}(w)} = \lim_{\rho \downarrow 0} \pv(w)
\]
is optimal to the order of $O(\rho^2)$ as $\rho \to 0$.
\end{cor}

\begin{proof}
Consider the semi-linear, second-order ODE
\begin{equation*}
{\rm{G}}(w, u, u_w, u_{ww}) = \la u + \rho(u - 1) {\bf 1}_{\{ w < 0 \}} - \big[(rw-c) + (\mu-r) \pi_L \big] u_w - \frac{1}{2} \sig^2 \pi_L^2 u_{ww} = 0,
\end{equation*}
with boundary conditions $u(-L) = \rho/(\la + \rho)$ and $u(c/r) = 0$.  As in Bayraktar and Zhang \cite{4.}, one can show that there exists a unique solution $\psi_L$ of this BVP, and by a verification theorem similar to Theorem \ref{thm:verif}, one can show that $\psi_L$ equals the (L-approximate) probability of lifetime exponential Parisian ruin associated with the investment strategy given by $\pi_L$.  Similarly, one can prove a verification lemma for ${\rm{G}}$
on $]-L, c/r[\,$.

One can show that
\[
{\rm{G}}\big(w,  \rho m,  \rho m_w, \rho m_{ww} \big) = \rho^2 m(w) {\bf 1}_{\{ w < 0 \}} > 0,
\]
for $-L < w < c/r$. Also,
\[
\psi_L(-L) = \dfrac{\rho}{\la + \rho} < \dfrac{\rho}{\la} = \rho m(-L),
\]
and
\[
\psi_L(c/r) = 0 = \rho m(c/r).
\]
Thus, because ${\rm{G}}(w, \psi_L, (\psi_L)_w, (\psi_L)_{ww}) = 0$, we deduce that $\psi_L(w) < \rho m(w)$ for all \\
 $-L \le w < c/r$.  Similarly,
\[
{\rm{G}}\big(w,  \rho m - (\rho/\la)^2,  \rho m_w, \rho m_{ww} \big) < 0,
\]
for $-L < w < c/r$.  Also, by \eqref{eq:4_13} and \eqref{eq:4_14},
\[
\rho m(-L) - \left( \dfrac{\rho}{\la} \right)^2 = \dfrac{\rho}{\la} -  \left( \dfrac{\rho}{\la} \right)^2 <  \dfrac{\rho}{\la + \rho} = \psi_L(-L),
\]
and
\[
\rho m(c/r) - \left( \dfrac{\rho}{\la} \right)^2 = - \left( \dfrac{\rho}{\la} \right)^2 < 0 = \psi_L(c/r).
\]
Thus, because ${\rm{G}}(w, \psi_L, (\psi_L)_w, (\psi_L)_{ww}) = 0$, we deduce that $\rho m(w) - (\rho/\la)^2 < \psi_L(w)$ for all $-L \le w \le c/r$.

We have shown that
\[
\rho m(w) - \left( \dfrac{\rho}{\la} \right)^2 < \psi_L(w) < \rho m(w),
\]
for all $-L \le w < c/r$.  Recall from the proof of Proposition \ref{prop:psil} that
\[
\rho m(w) - \left( \dfrac{\rho}{\la} \right)^2 < \psi(w) < \rho m(w),
\]
for all $-L \le w < c/r$.  By combining the above two sets of inequalities, we obtain
\[
\psi(w) - \left( \dfrac{\rho}{\la} \right)^2 < \psi_L(w) < \psi(w) + \left( \dfrac{\rho}{\la} \right)^2,
\]
for all $-L \le w \le c/r$, which proves the corollary.
\end{proof}

Proposition \ref{prop:psil} and Corollary \ref{cor:piL} help us to understand the observations preceding Proposition \ref{prop:pi_rho}, namely, that $\lim_{\rho \downarrow 0} \psi \equiv 0$ but $\lim_{\rho \downarrow 0} \pv = \pi_L$.  Furthermore, from Section 6.2 in Landriault et al.\ \cite{13.}, we know that the probability of exponential Parisian ruin equals
\[
\P(\kap^\pi < \infty) = 1 - \E \Big( e^{-\rho \int_0^\infty {\bf 1}_{\{ W^\pi_t < 0\}} \dd t} \Big);
\]
thus, as $\rho \downarrow 0$, we see that this expression goes to $0$ regardless of the time of death of the individual, which also implies that the probability of {\it lifetime} exponential Parisian ruin goes to $0$ as $\rho \downarrow 0$.

\section{Conclusions}\label{sec:5}

In this paper, we (approximately) minimized the probability of lifetime exponential Parisian ruin, that is, the probability that wealth stays below zero in excess of an exponentially distributed clock and before the individual dies.  Most of the related research we found  calculates the probability of Parisian ruin; our work is the first to {\it control} the probability of Parisian ruin.

We proved some interesting properties of the optimal investment strategy $\pv$.  For example, we proved that as the rate of the exponential clock controlling Parisian ruin $\rho$ increases, the optimal investment strategy remains the same for positive wealth and strictly increases for negative wealth.  The independence of the optimal investment strategy on $\rho$ when wealth is positive shows a myopia that we commonly see in such goal-seeking problems.  We also analyzed how $\pv$ changes with $L$, and $\pv$'s behavior depends on how $r$ is related to the other parameters.

We also obtained an asymptotic expansion of the minimum probability of lifetime exponential Parisian ruin for small values of the hazard rate of the excursion clock and found a close connection between our value function $\psi$ and the minimum expected occupation time $m$ from Bayraktar and Young \cite{8.}.  Specifically, we proved that $\psi = \rho m + O(\rho^2)$ uniformly on $[-L, c/r]$.  We also proved that using the investment strategy corresponding to $m$'s problem yields a value function that is within order $O(\rho^2)$ of $\psi$.

In future work, we will add life insurance and life annuity products to our model and will analyze the effects of Parisian ruin on the investment, life insurance, and annuity purchasing strategies, as compared with the corresponding strategies when minimizing the probability of ordinary ruin.  Another interesting application of exponential Parisian ruin is to consider stochastic control problems with non-life insurance models, especially from the standpoint of an insurance company purchasing reinsurance; see, for example, our working paper \cite{19.}.

\begin{acknowledgements}
X. Liang thanks the National Natural Science Foundation of China (11701139, 11571189) and the Natural Science Foundation
of Hebei Province (A2018202057) for financial support of her research. V.R. Young thanks the Cecil J. and Ethel M. Nesbitt Professorship for financial support of her research.
\end{acknowledgements}


\end{document}